\documentclass[12pt,reqno]{amsart}\usepackage{amssymb,amsmath,amscd,epsfig,amsthm,enumerate,import,eufrak}

\newtheorem{theorem}{Theorem}[section]
\newtheorem{proposition}[theorem]{Proposition}
\newtheorem{lemma}[theorem]{Lemma}

\theoremstyle{definition}
 
{

\theoremstyle{remark}

\numberwithin{equation}{section}

\usepackage{epsfig,amsthm,enumerate,import}

\theoremstyle{plain}

\theoremstyle{definition}

\def \r {\mathbb{R}}

\def \c {\mathbb{C}}

\def \p {\mathbb{P}}

\def \q {\mathcal{Q}}

\def \u {\mathcal{U}}

\def \f {\mathcal{F}}
\def \h {\mathbb{H}}

\def \mcp {\mathcal{P}}

\def \mfg {\mathfrak{g}}

\def \vs {\vskip 2mm}

\def \dd[#1,#2]{\frac{\partial #1}{\partial #2}}

\def \D {\triangle}
\def \HH {\mathcal{H}}
\def \and {\; \text{and}\;}

\def \compvar[#1,#2]{\HH^{k,\alpha}_{#1, #2}(\D, \u_{#2})}

\def \mod {\; \text{mod}\;}

\def \naba {\nabla^A}
\def \dbdth {\frac{\partial \;\;}{\partial \theta}}

\theoremstyle{remark}

\numberwithin{theorem}{section}
\makeatletter
\def\v@rt#1#2{\m@th\ooalign{$\hfil#1|\hfil$\crcr$#1#2$}}
\def\captr{\mathrel{\mathpalette\v@rt\cap}}
\makeatother
\parindent=0in
\begin{document}
\title[Left-invariant Tubes on $SU(2)$]{Left-invariant Grauert Tubes on $SU(2)$}
\author[Vaqaas Aslam]{Vaqaas Aslam} 
\author[Daniel Burns Jr.]{Daniel M Burns Jr.}
\author[Daniel Irvine]{Daniel Irvine}
\address{Department of Mathematics, University of Michigan, Ann Arbor, MI 48109-1043, USA.}
\email{svaslam@umich.edu}
\email{dburns@umich.edu}
\email{dirvine@umich.edu}

\begin{abstract} Let $M$ be a real analytic Riemannian manifold. An adapted complex structure on $TM$ is a complex structure on a neighborhood of the zero section such that the leaves of the Riemann foliation are complex submanifolds. This structure is called entire if it may be extended to the whole of $TM$. We call such manifolds Grauert tubes, or simply tubes. We consider here the case of $M = G$ a compact connected Lie group with a left-invariant metric, and try to determine for which such metrics the associated tube is entire. It is well-known that the Grauert tube of a bi-invariant metric on a Lie group is entire. The case of the smallest group  $SU(2)$ is treated completely, thanks to the complete integrability of the geodesic flow for such a metric, a standard result in classical mechanics. Along the way we find a new obstruction to tubes being entire which is made visible by the complete integrability.

\end{abstract}

\maketitle
%
%
%
%
%
%
%
%
%
%
%
%
\section{Introduction}\label{sec:intro}
A real analytic manifold $M$ has a natural complexification to a germ of complex manifold $M_{\c}$ together with a conjugation fixing the original manifold. If we are given a real analytic Riemannian metric $g$ on $M$, we get a canonical complex structure, called the {\em adapted complex structure}, on the tangent bundle $TM$ in a neighborhood of the zero section $0 \subset TM$ such that it verifies several compatibility conditions with $(M, g)$. Such a complex manifold we call a Grauert tube, or simply tube, and say it is entire if the adapted complex structure can be extended to all of $TM$. Among the compatible properties of the tube structure is that the length-squared function $\tau = \|v\|^2, v \in TM$ is a strictly plurisubharmonic function on the tube, and its square root $u =\sqrt{\tau}$ is a solution of the homogeneous complex Monge-Amp\`ere equation (HCMA) on the complement of the zero section, whose associated characteristic foliation is the geodesic foliation, or Riemann foliation, in $TM$. See \cite{ls, s, gs1, gs2}.

\vs

Entire tubes appear to be rare, and all constructions of them so far are closely related to homogeneous examples \cite{a1, s}. Lempert and Sz\H{o}ke \cite{ls} have shown that a necessary condition for the tube of $(M, g)$ to be entire is that the sectional curvatures of $g$ be non-negative. This is far from sufficient, however. Sz\H{o}ke has shown that there is only a two parameter family of metrics of revolution on $S^2$ which have entire tubes, whereas many of these manifolds have strictly positive curvature. Furthermore, for $M$ of two dimensions, Aguilar \cite{a2} has found an infinite sequence of conditions generalizing that of Lempert-Sz\H{o}ke to higher order invariants which give necessary and sufficient conditions that a tube be entire. Unfortunately, these are very difficult to interpret in specific examples. In particular, it is not known whether the conditions are effectively finite, i.e., are redundant after a certain degree. One of the problems in estimating whether a given metric has an entire tube is that the complex structure on the tube viewed as a subset of the tangent bundle $TM$ is not very explicit. Sometimes the original real manifold may have an obvious ``large" complexification which presents itself as a natural guess for an entire tube complexification, but it is hard to say how far out in such a complex manifold the tube's structures extend: for example, the solution $u$ of HCMA. 

\vs

If $M = G$, a compact, connected Lie group, then there is a natural ``large" complexification, namely the complex group $G_{\c}$. If the metric $g$ is bi-invariant on $G$, then Sz\H{o}ke \cite{s} has shown that the tube is entire and equal to $G_{\c}$. On the other hand, it is well-known that some left-invariant metrics on $S^3 = SU(2)$ have scalar curvature $R < 0$, which shows they cannot have entire tubes, by Lempert-Sz\H{o}ke's condition. We try to address here the question of whether a left-invariant metric has an entire Grauert tube. Our first lemmas state that if an invariant metric has an entire tube, then the tube is indeed biholomorphic to $G_{\c}$. Given this, to see whether the tube is entire, we can work on the complex group $G_{\c}$ and try to determine whether we can follow any of the structure on the tube which would come from its identification with $TM$ intrinsically on the complex group, and  decide whether these complex geometric features present obstructions to the existence, e.g., of the solution $u$ of HCMA on $G_{\c}$. The existence of $u$ implies that $G_{\c} \setminus G$ is foliated by complex curves which are the complexifications of the real $g$-geodesics on $G$. But we can complexify the metric $g$ to a holomorphic complex metric $g_{\c}$ on $G_{\c}$, and consider its holomorphic geodesic flow on $T^{1,0}G_{\c}$, the holomorphic tangent bundle of $G_{\c}$. The complex conjugation on $G_{\c}$ acts on everything here, and we are left with the complex geometric question: do the complexified real geodesics foliate $G_{\c} \setminus G$? An interesting feature of these examples studied here is that, in general, there are two independent sorts of obstructions to extending a tube: {\em geometric} ones related to curvature and focusing, as in section \ref{sec:U2cases}, and intrinsic complex {\em analytic} questions of convergence, related to the incompleteness of the geodesic flow in the holomorphic case, as in section \ref{sec:complexintegrals}. 
\vs

It is a well-known observation of Arnol'd that the geodesic flow for a left-invariant metric on a compact Lie group is equivalent to the motion of a free rigid body whose frame is parametrized by $G$ about its center of mass (see \cite{via}, Appendix 2, for example). At this point we specialize to the case of $G = SU(2)$ (or $SO(3)$), because in this case, the equations of motion of the free rigid body in $\r^3$ are completely integrable, and the first integrals are simply quadratic in convenient algebraic coordinates. Our approach here is to complexify the integrals of the free rigid body to the space $T^{1,0}G_{\c}$ to study complex geodesics on $G_{\c}$. Complete integrability enables us to draw enough information about the geodesic flow in these cases to settle whether the Grauert tubes are entire or not. 
\vs
More specifically, the sections are divided as follows. In section \ref{sec:GTlemmas} we review quickly some basics of Grauert tubes, and recall from \cite{s3} the result that if a compact group $G$ of isometries acts on the real analytic Riemannian manifold $(M, g)$, and the Grauert tube $M_{\c}$ of $(M, g)$ is entire, then $G_{\c}$ acts holomorphically on the entire tube. (The authors thank the referee for pointing out \cite{s3} for this result.) In particular, if $G$ acts transitively on $M$, with isotropy group $H$ at a reference point $z_0 \in M$, then the entire Grauert tube is biholomorphic to $G_{\c}/H_{\c}.$ We make some reference to finite radius versions of these statements, but the details are largely left to the reader.

\vs
In section {\ref{sec:Gcgeom}} we consider the relation of the real geometry of $G, g$ and the holomorphic geometry of $G_{\c}, g_{\c}$. We recast the criterion of \cite{ls} for existence of the adapted complex structure into statements about the exponential map $Exp_{\c}$ of the holomorphic metric $g_{\c}$ when restricted to the imaginary points $i TG \subset T^{1,0}G_{\c}$. There are two issues here: is the map defined on all of $i TG$ (this will be treated in sections \ref{sec:complexintegrals}, \ref{sec:U2cases}) and in that case, is the map a diffeomorphism? The differential is computed in terms of Jacobi fields and their complexifications.
\vs
In section \ref{sec:complexintegrals} we recall classical facts about the motion of a free rigid body represented by a compact Lie group $G$ and its equivalence to the geodesic flow for a left-invariant metric on $G$, in particular, we recall the complete integrability of such a system for $G = SU(2)$. For a generic left-invariant metric one can show, using the mechanical integrals of motion, that the map $Exp_{\c}$ is not well-defined on all of $i TG$ because most complexified (real) geodesics are not given by entire functions into $G_{\c}$. The remaining cases all have an extra symmetry, i.e., they are homogeneous under $U(2) = SU(2) \times S^1/ <\pm 1 >$.
\vs
In section \ref{sec:U2cases} we treat the cases on $SU(2)$ with extra symmetry. These are sometimes referred to as Berger spheres. These metrics are parametrized by a parameter $\lambda$ in the positive real line, where $\lambda = 1$ is the case of the round (unit) sphere. We first show that all $g_{\lambda},$ with $\lambda \in (0,1]$ have entire tubes. This is done by exhibiting such metrics as Riemannian quotients and invoking a result of Sz\H{o}ke \cite{s2}, or alternatively, of Aguilar \cite{a1}. Then we show that the metrics $g_{\lambda},$ with $\lambda > 1$, have finite radius tubes by solving the Jacobi equations to calculate the behavior of the differential of the map $Exp_{\c}$. For such metrics, the map $Exp_{\c}$ is defined and smooth on $i TG$, but singularities of its differential occur. We solve the Jacobi equations for certain geodesics on such manifolds and prove the existence of focal points in finite time, giving an upper bound to the radius of existence of the Grauert tube, a bound which tends to $+\infty$ as $\lambda \searrow 1$.
\vs
In section \ref{sec:rmks} we just include a few remarks and open questions remaining after the present work. 
\vs
For convenience and completeness, we have collected some elementary and explicit differential geometric details on left-invariant metrics on $SU(2)$ in an appendix which appears as section 7 of the arXiv posting of this paper, \cite{web}.
\vs
This work was supported by the US National Science Foundation [DMS-1105586, RTG grant 1045119].

\vs\vs
%
%
%
%
%
%
%
%
%
%
%
%
\section{Grauert Tubes and Some Simple Lemmas}\label{sec:GTlemmas}
\vs\vs
We review briefly some basic properties of the construction of adapted complex structures on neighborhoods of the zero section in the tangent bundle $TM$ \cite{ls,s} or equivalently, cotangent bundle $T^*M$ \cite{gs1, gs2} of a compact, real analytic Riemannian manifold $(M, g)$. For specificity we will follow the covariant (tangent bundle) formulation of \cite{ls,s}.
\vs
Let $\pi: TM \to M$ denote the projection, and let $\Theta$ denote the canonical one-form on $TM$ given by 
\[
	\Theta(V) := (D\pi(V), v),
\]
where $v \in T_{\pi(v)}M, V \in T_v(TM)$, and $(\cdot,\cdot)$ denotes the Riemannian metric $g$. An {\em adapted complex structure} on $T^RM := \{v \in TM \, | \, |v| < R\}$ is a complex structure for which the leaves of the Riemann foliation $\mathcal{F}$ of $TM \setminus 0_M$ are complex submanifolds. Recall that the Riemann foliation is given by the tangent sub-bundles to all geodesics for the metric $g$ as leaves. Thus, the adapted complex structure on $T\r$ identifies it with $\c$ via 
\[
	T\r \ni (s, t \frac{d\,\,}{ds}) \to s+ i t \in \c,
\]
and is functorial for geodesic immersions. 
\vs
For the maximal radius $R$ for which the adapted complex structure is defined, the complex manifold $T^RM$ is called the {\em Grauert tube} of $(M, g)$. If $R = +\infty$, we say the tube is {\em entire}.  
\vs
Here are some properties of tubes which will be needed in what follows:
\vs
\begin{enumerate}
\item If $\phi$ is an isometry of $(M, g)$, then its action by its differential $D\phi$ on the Grauert tube $T^RM$ is holomorphic for the adapted complex structure, while the map $v \to -v$ is anti-holomorphic.
\vs
\item The function $r^2 := |v|^2$ given by the $g$-length squared of vectors is strictly plurisubharmonic on the tube.
\vs
\item We have that  \begin{equation} \label{eqn:Liouville}
	\frac{i}2 \partial \bar{\partial} r^2 = d\Theta := \Omega
\end{equation}
is the Liouville form on $TM$ and is a K\"ahler form. The corresponding K\"ahler metric induces the original metric $g$ along the zero section $M$. 
\vs
%
%
%
%
%
%
%
%
\vs
\item $r$ satisfies the homogeneous complex Monge-Amp\`ere equation $\partial \bar{\partial} u^n = 0$ on $T^RM \setminus M$.
\end{enumerate}
\vs
Let $G$ be the isometry group of $(M, g)$, and $\mathfrak{g}$ its Lie algebra. An element $\xi \in \mathfrak{g}$ will be identified with its induced vectorfield on $M$. For any vectorfield $\xi$ we will denote by $\Phi^{\xi}_t$ the corresponding 1-parameter group of diffeomorphisms. This group lifts to an action by differentials on $TM$, and the induced vectorfield on $TM$ will be denoted $\tilde{\xi}$. In particular, $D\pi\tilde{\xi} = \xi$. If $(q^1, \ldots, q^n)$ are local coordinates on $M$, and $\xi = \sum_k a^k(q) \frac{\partial \,\,}{\partial q^k},$ then in the corresponding coordinates $(q^1, \ldots, q^n, v^1, \ldots, v^n)$ the field $\tilde{\xi}$ can be expressed as follows:
\[
	\tilde{\xi} = \sum_k a^k(q) \frac{\partial \,\,}{\partial q^k} - \left(\sum_ j v^i \frac{\partial a^k(q)}{\partial q^i}\right) \frac{\partial \,\,}{\partial v^k}.
\]
It will be important to note that 
\begin{equation} \label{eqn:scale}
	(N_{\rho})_* \, \tilde{\xi} = \tilde{\xi},
\end{equation}
for all $\rho \neq 0.$
\vs
It follows from the property (1) above for Grauert tubes that $\Phi^{\tilde{\xi}}_t$ acts biholomorphically on the tube, and thus
\[
	\tilde{\xi}^{1,0} = \frac12 (\tilde{\xi} - i J \tilde{\xi})
\]
is a holomorphic vectorfield on the tube, where $J$ is the almost complex structure on the tube. 
\vs
Finally, let $G_{\c}$ denote the complexification of the group $G$, with Lie algebra $\mathfrak{g}_{\c} = \mathfrak{g} \otimes \c$. We can now state the first lemma, which is due to Sz\H{o}ke (\cite{s3}, Theorem C).
\vs
\begin{lemma} \label{lemma:action}
Let $(M, g)$ be a compact Riemannian manifold with entire Grauert tube. Then the complex group $G_{\c}$ acts holomorphically on all of the tube, i.e., on all of $TM$ with the adapted complex structure.
\end{lemma}
\vs

This, in turn, is the main ingredient in the proof of the next lemma.

\vs
\begin{lemma} \label{lem:homogeneity}
	Let $M = G,$ a compact, connected Lie group, and $g$ be a left invariant metric on $M$. Assume that the Grauert tube of $M$ is entire. Then the 	connected complex group $G_{\c}$ acts simply transitively on $M_{\c} := TM$, i.e., the entire tube $TM$ is biholomorphic to $G_{\c}$.
	\end{lemma}
\vs
\begin{proof}
By lemma \ref{lemma:action}, we know that $G_{\c}$ acts on $M_{\c}$. We will prove that it acts locally transitively at every point $v \in M_{\c}$. Assuming this for the moment, this proves that all orbits of $G_{\c}$ on $M_{\c}$ are open, and since $M_{\c}$ is connected, this proves that the mapping $A_v: G_{\c} \ni h \to h\cdot v \in M_{\c}$, for any fixed $v \in M_{\c}$, is surjective. Since $G_{\c}, M_{\c}$ are of the same complex dimension, and the map $A_v$ is $G_{\c}$ equivariant, Sard's theorem says that every point in $M_{\c}$ is a regular value of $A_v$, and hence, $A_v$ is a topological covering. If $v \in M$, the covering $A_v: G \to M$ is a diffeomorphism. Since the inclusions $G \subset G_{\c}, M \subset M_{\c}$ are homotopy equivalences, the covering $A_v$ has degree 1 over all of $M_{\c}$, and therefore $A_v$ is a global biholomorphism.
\vs
\vs
To show that $G_{\c}$ acts locally transitively, it suffices to show that 
\[
	DA_{v}: T^{1,0}_{I} G_{\c} = \mathfrak{g}_{\c} \to T_v^{1,0} M_{\c}
\]
is an isomorphism. This, in turn is equivalent to showing that, for a basis $\xi_1, \xi_2,\ldots, \xi_n$ of $\mathfrak{g}$ over $\r$, the complex fields $\tilde{\xi}^{1,0}_1, \tilde{\xi}^{1,0}_2, \ldots, \tilde{\xi}^{1,0}_n$ are a basis for $T^{1,0}M_{\c}$ over $\c$ at $v$. Define $\mcp = \{ \tilde{\xi}^{1,0}_1 \wedge \tilde{\xi}^{1,0}_2 \wedge \ldots \wedge \tilde{\xi}^{1,0}_n = 0\}$. We wish to show $\mcp$ is empty. 
\vs
Assume it is not empty. Then $\mcp$ is a $G_{\c}$-invariant divisor on $M_{\c}$, and $\mcp \cap M = \emptyset$. Consider the plurisubharmonic function $u$ restricted to $\mcp$. Since $u$ is non-negative and proper when restricted to $\mcp$, it achieves a global minimum at some $v \in \mcp$. The map $A_v$ restricted to $G$ gives an embedding of $G$ into $\mcp \subset M_{\c}$ where it is a $CR$-submanifold of real dimension $n$. Note that the function $u$ is constant on $S_v := A_v(G)$. Since $\dim_{\c} \mcp < n$, $S_v$ cannot be a totally real submanifold in $\mcp$. Hence its complex tangent space is positive dimensional at every point. Its Levi-form cannot be trivial, since otherwise the leaves of the Levi foliation of $S_v$ would give positive-dimensional complex submanifolds in the level set $\{u = u(v)\}$, which is strictly pseudoconvex, by the tube construction. By a classical construction in complex analysis (local Bishop disks), there exists a continuous mapping $F: \bar{\triangle} \to U \subset M_{\c}$, where $U$ is a small polydisk neighborhood of $v$, and with $F|_{\triangle}$ holomorphic and non-constant, and $F(\partial \triangle) \subset S_v$. Thus $u\circ F|_{\partial\triangle} \equiv u(v)$. Therefore, $F(\triangle) \subset \mcp \cap U$, since $\mcp \cap U$ is closed and analytic in $U$, which gives us that $u\circ F(z) \geq u(v),$ for all $z \in \triangle$. By the maximum principle, $u\circ F$ is constant on $\triangle$, contradicting, again, the strict pseudoconvexity of $S_v$. 
\end{proof}
\vs\vs
%
%
%
%
%
%
%
%
%
%
%
%
%
%
\section{The holomorphic geometry of $G_{\c}$} \label{sec:Gcgeom}
\vs\vs
In this section, we relate the Grauert tube over a left-invariant metric on $G$ and the holomorphic geometry of its extension to all of $G_{\c}$. In this way, we will be able to study questions about the possible radius of such tubes by corresponding properties of the holomorphic Riemannian geometry of $G_{\c}$, which is always a $G_{\c}$ homogeneous structure on all of $G_{\c}$.
\vs\vs
Let $T^{1,0}G_{\c}$ be the holomorphic tangent bundle of $G_{\c}$ which is biholomorphic to $G_{\c} \times \mathfrak{g}_{\c},$ an explicit isomorphism given by left-invariant vectorfields on $G_{\c}$. Let $\sigma: G_{\c} \to G_{\c}$ be the complex conjugation such that Fix($\sigma) = G$. Its differential is $D\sigma: T^{1,0}G_{\c} \to T^{0,1}G_{\c}$. This induces a conjugation $\tilde{\sigma}$ on $T^{1,0}G_{\c}$ given by
\begin{equation} \label{eqn:conjdef}
	T^{1,0}G_{\c} \ni \zeta \to \tilde{\sigma}(\zeta) := \overline{D\sigma(\zeta)} \in T^{1,0}G_{\c}.
\end{equation}
Then we will identify Fix($\tilde{\sigma}$) = $TG$, and for a section $\zeta$ of $T^{1,0}G_{\c}$ over $G$, we have $\tilde{\sigma}(\zeta) = \zeta$ if and only if $\zeta$ is in the \emph{real} span of $\tilde{\xi}^{1,0}$, with $\xi \in \mfg$ as above, over each $g \in G$. Equivalently, $\zeta(g) \in \{g\} \times \mathfrak{g} \subset G_{\c} \times \mathfrak{g}_{\c}.$ To fix ideas, if $G = \r$ (at this point we do not need to assume our group $G$ is compact) with the Euclidean metric and unit tangent vector $\frac{d}{dx}$, so that $\mfg \cong \r \frac{d}{dx} $, then $G_{\c}\cong \c$ with coordinate $z = x+iy$, and $\widetilde{(\frac{d}{dx})}^{1,0} = \frac{\partial}{\partial z}$ on $\c$. In particular, $\frac{\partial}{\partial z}$ is real over $\r = G \subset G_{\c} \cong \c$, and a section $f(z) \frac{\partial}{\partial z}$ of $T^{1,0}G_{\c}$ is $\tilde{\sigma}$-real if and only if $\overline{f(\bar{z})} = f(z)$. 
\vs
A left invariant metric on $M = G$ is given by a symmetric, positive definite transformation $A: \mathfrak{g} \to \mathfrak{g}$, where we measure symmetry with respect to the inner product $(\xi, \eta) = - B(\xi, \eta)$, $B$ being the Killing form on $\mathfrak{g}$. If $\mfg$ is simple, the form $-B$ determines the unique (up to constant positive scale) bi-invariant metric on $G$, which is just the round metric of sectional curvature $\frac18$ when $G \cong SU(2) \cong S^3$. We extend $A$ to $\mathfrak{g}_{\c}$ by extending scalars, and similarly for $B$ which gives $B_{\c}$, the Killing form for $\mathfrak{g}_{\c}$. This gives a non-degenerate, holomorphic bilinear form on $T^{1,0}G_{\c}$. This form, when restricted to the real sub-bundle $TG \subset T^{1,0}G_{\c}$, induces the original metric on $G$. Otherwise put, every left-invariant metric on $G$ has a unique holomorphic extension to a holomorphic metric on $G_{\c}$. We will denote the left-invariant metric associated to $A$, when we have to be precise, by $g_A$, and the complexified holomorphic version of this metric by $g_{A, \c}$. Using $g_{A, \c}$ we get an induced symplectic structure on $T^{1,0}G_{\c}$, with form $\Omega_{\c}$. This form restricts to $\Omega$ on $TG$. We have the induced geodesic flow on $T^{1,0}G_{\c}$, though it is not necessarily complete in the complex case. We have a kind of normal exponential map, a smooth mapping $\Phi: TG \to G_{\c}$ given by 
\begin{equation}  \label{eqn:Expdef}
	\Phi: TG \ni v \to \gamma_{\c}(i v) = Exp_{\c}(i v) \in G_{\c},
\end{equation} 
where, for $v \neq 0$, $\gamma$ is the real geodesic with $\gamma(0) = \pi(v), \dot{\gamma}(0) = \frac{v}{|v|}$. For $v = 0 \in T_gG, \Phi(v) = g$. This map is, at first, just defined in a neighborhood of the zero section $G$. It is real analytic, $G$-equivariant and sends a real geodesic $\gamma \subset G$ to the real points of the corresponding $g_{\c}$-geodesic. Now if we consider $TG$ with its adapted complex structure, which we will momentarily call $J_T$, and we consider the complex submanifold $T\gamma \subset TM, J_T$, for $\gamma$ a $g$-geodesic, then $\Phi$ restricted to $i T\gamma$ maps it biholomorphically onto a neighborhood of $\gamma \subset \gamma_{\c} \subset G_{\c}$. Since this is true for all geodesics in $G, \Phi$ is holomorphic from $TM, J_T$ to $G_{\c}$. But then $\Phi$ is locally $G_{\c}$-equivariant, and is, locally near the zero section, the inverse of the action map from $G_{\c} \to TM, J_T,$ which is a global biholomorphism, by lemma \ref{lem:homogeneity}. Hence, $\Phi$ is globally defined by analytic extension, and is holomorphic. It embeds every $T\gamma \subset TM, J_T$ onto $\gamma_{\c} \subset G_{\c}$ biholomorphically. We summarize this discussion in the following lemma. We call a connected, $\sigma$-invariant complex geodesic curve $\Gamma \subset G_{\c}$ a {\em real} $g_{\c}$-geodesic if $\Gamma$ is the complexification of a real geodesic $\gamma \subset G$.
\begin{lemma}  \label{lem:folGc}
Let $g = g_A$ be a left invariant metric on $G$ such that its Grauert tube is entire. Then all real $g_{\c}$-geodesics on $G_{\c}$ for the holomorphic metric $g_{\c} = g_{A, \c}$ are complete. Furthermore, they give a foliation $\mathcal{F}_{\c}$ of $G_{\c} \setminus G$.
\end{lemma}
\qed
\vs
Next, if we assume $\Phi$ is defined on $TG$ on a neighborhood $|v|^2_g < R$, we calculate the differential of $\Phi$. First, there are the directions in $T_v(TG)$ spanned by the vectors $\tilde{\eta}_x, x \in \mfg$. Since $\Phi$ is $G$-equivariant, 
\vs
\[
	D\Phi_{v}(\tilde{\eta}_x)(v) = \eta^{1,0}_{x,\c}(\Phi(v)).
\]
\vs
Thus $D\Phi$ is always injective on this subspace. A complementary subspace in $T_v(TG)$ is given the vertical directions $VT_v \subset T_v(TG)$, tangent to the fiber through $v$. This is identified naturally with $T_{\pi(v)}G$, where $\pi(v)$ is the projection of $v$ in $G$. For a fixed $w \in T_{\pi(v)}G,$ we have a vector, still denoted $w$, in every $VT_{v'}(TG),$ where $\pi(v) =\pi(v').$ For the real exponential map $Exp: TG \to G$, we have that
\begin{equation}\label{eqn:DExp}
	DExp_{t v}(w) = Y_w(\gamma(t)),
\end{equation}
where $\gamma(t)$ is the (real) geodesic satisfying $\gamma(0) = \pi(v), \dot{\gamma}(0) = v,$ and $Y_w$ is the Jacobi field along $\gamma$ with initial conditions $Y_w(\pi(v)) = 0, \nabla_{\dot{\gamma}}Y_w(\pi(v)) = w.$ Letting $\zeta = t + i s \in \c$, we may extend $\gamma, Y_w$ holomorphically to $\gamma_{\c}(\zeta),$ $Y_{w, *}(\gamma(\zeta)).$ Similarly, in the holomorphic geometry of $G_{\c}$,we can consider the holomorphic geodesic $\Gamma(\zeta)$ satisfying 
\vs
\begin{equation}\label{eqn:Gamma}
	\nabla^{\c}_{\dot{\Gamma}(\zeta)} \dot{\Gamma}(\zeta) = 0, \; \Gamma(0) =\pi(v), \; \dot{\Gamma}(0) = i v \in T^{1,0}_{\pi(v)}(G_{\c}).
\end{equation}
\vs
and the holomorphic Jacobi field $W^{\c}_w = W^{\c}_w(\Gamma(\zeta)$ along $\Gamma$ satisfying $W_w^{\c}(\pi(v))$ $= 0, \nabla^{\c}_{\dot{\Gamma}} W_w^{\c}(\pi(v)) = w.$ Then we have the holomorphic case of (\ref{eqn:DExp}):
\vs\vs
\begin{equation}\label{eqn:DExpC}
	DExp_{\c, i t v }(w) = W^{\c}_w(\Gamma(t)) \in T^{1,0}_{\Gamma(t)}(G_{\c}).
\end{equation}
\vs\vs
\begin{lemma}\label{lem:ICs} The following identities hold:
\[
	\begin{array}{l} \text{(1)}\; \Gamma(\zeta) = \gamma_{\c}(i \zeta) \\
	\\
	\text{(2)} \; W^{\c}_w(\Gamma(\zeta)) = Y^{\c}(\gamma_{\c}(i \zeta))\\
	\\
	\text{(3)} \; D\Phi_{t v}(w) = Y^{\c}(\gamma_{\c}(i t v)).
	\end{array}
\]
\end{lemma}
\vs\vs
(1) and (2) follow from the holomorphy of $\gamma_{\c}, Y^{\c}_w, \Gamma,$ and $W^{\c}_w$, and the uniqueness theorem for holomorphic ODE.  (3) is simply (\ref{eqn:DExpC}) taking into account the definition of $\Phi$ (\ref{eqn:Expdef}), and parts (1) and (2). \\
%
\qed
\vs\vs
It will be convenient later (section \ref{sec:U2cases}) to compute the rank of $D\Phi$ calculating Jacobi fields in terms of the left invariant parallelism on $TG$, whereas it is easy to compute the differential $D\Phi$ acting on right invariant fields. For purposes of comparison, we will need to recall the well-known formula relating the two. If $X \in \mfg \cong T_eG$, denote by $\eta_X(g) =DR_{g}(X) \in T_gG$ its right-invariant extension over all of $G$, and $\xi_X(g) = DL_{g} X \in T_gG$ its left-invariant extension. Then one has
\vs
	\begin{equation}\label{eqn:xi2eta}
		\eta_X(g) = \xi_{Ad(g^{-1})X}(g) \in T_gG.
	\end{equation}
\vs\vs
A similar equation holds for $G_{\c}$ and $\mfg_{\c}$. We record here finally the general computation of the rank of $D\Phi_{t v}$. Let $X_1,\ldots, X_d$ be a real basis of $\mfg$ and let $Y_{X_i}$ be the Jacobi fields as in (\ref{eqn:DExp}) for $w = \eta_{X_i} \in T_{\pi(v)}G.$ Express the $Y_i$ in the basis $\eta_{X-i}$,so that one has
\vs
\begin{equation}\label{eqn:eta2Y}
	Y^{\c}_i(\gamma_{\c}(\zeta)) = a_{i, 1}(\zeta) \; \eta^{1,0}_{X_1,\c}(\gamma_{\c}(\zeta)) + \ldots + a_{i, d}(\zeta) \; \eta^{1,0}_{X_d, \c}(\gamma_{\c}(\zeta)).
\end{equation}
\vs\vs
Assembling this, we have a numerical statement summarizing this discussion as follows.
\vs
\begin{lemma}\label{lem:rankDPhi} The rank of $D\Phi_{t v}$ is given by
	\begin{equation}\label{eqn:rank}
		\text{rank}\; D\Phi_{t v} = d\, + \;{\text{rank}}\; {\; \text{Im} \;} \left(\begin{array}{ccc} a_{1,1}(i t) & \cdots & a_{1,d}(i t) \\ && \\ & \cdots & \\ && \\ a_{d, 1}(i t) & \cdots & a_{d,d}(i t) \end{array} \right),
	\end{equation}
where the rank of $D\Phi_{t v}$ means the real rank of the differential, and $Im$ is the imaginary part.
\end{lemma}
\qed
\vs\vs\vs
%
%
%
%
%
%
%
%
\section{The Free Rigid Body and Complex Integrals}\label{sec:complexintegrals}
\vs
We will use, from Arnold \cite{via}, that the geodesic flow for a left-invariant metric on a compact Lie group is equivalent to the rotation of a rigid body with frame in $G$ about its center of mass. At this point we specialize to the case $G = SU(2)$, which has two explicit mechanical integrals.  Notation as above, at the beginning of section \ref{sec:Gcgeom}, these are the energy $E = (Av,v),$ and the total angular momentum $M = (A^2v,v)$. Taken together with any Hamiltonian generating a one-parameter group in $G$ acting on the tangent bundle shows that the geodesic system, with Hamiltonian $E$, is completely integrable. If we write the complex geodesic in $T^{1,0}G_{\c}$ as $\tilde{\gamma}(\zeta)  = (\gamma(\zeta), v(\zeta)) \in G_{\c} \times \mfg_{\c},$ with $\zeta \in \c$, the quadratic integrals are functions of $v(\zeta)$ alone, so that $E(v(\zeta)), M(v(\zeta))$ are constant along $\tilde{\gamma}$. If the complex geodesic is real, then the values $E(v(\zeta_0)) = a, M(v(\zeta_0)) = b$, are real and positive. The quadrics $\{ E(v) = a\}, \{M(v) = b\}$ are then non-singular for suitable $A$'s, $a$ and $b$, as are their projective completions $\q_{E,a}, \q_{M,b}$, {\em resp.}, in $\p^3 \supset \mfg_{\c}$. In homogeneous coordinates $z_1, z_2, z_3, w$, the quadrics $\q_{E,a}, \q_{M,b}$ are given by 
\vs
\begin{equation} \label{eqn:qe}
	0 = \lambda_1 \, z_1^2 + \lambda_2 \, z_2^2 + \lambda_3 \, z_3^2 - a \, w^2,
\end{equation}
\begin{equation} \label{eqn:qm}
	0 = \lambda_1^2 \, z_1^2 + \lambda_2^2 \, z_2^2 + \lambda_3^2 \, z_3^2 - b \, w^2,
\end{equation}\vs
{\em respectively}. Here we have assumed, without loss of generality, that $\lambda_1, \lambda_2, \lambda_3 \in \r^+$ are the eigenvalues of $A$, and that $A$ has been diagonalized. One should observe that this can be done by the adjoint action of $SU(2)$ on its Lie algebra. If $[\lambda_1: \lambda_2: \lambda_3: a] \neq [\lambda_1^2: \lambda_2^2: \lambda_3^2: b]$, then the intersection of these two quadrics is a curve in $\p^3$ which we denote by $C_{a,b}$. Our analysis will focus on the properties of this curve. We first check the variables $a, b$ enjoy a certain independence. 
\vs
\begin{lemma} \label{lem:gpab}
	The set of real parameters $(x_1, x_2, x_3) \in \r^3$ such that $a = E(x)$ and $b = M(x)$ verify 
	\begin{equation} \label{eqn:ineq}
		b \neq \lambda_k \; a, \; k = 1, 2, 3,
	\end{equation}
has non-empty interior.
\end{lemma}
We will call $a, b$ {\em generic} if $a, b$ satisfy the conclusion of the lemma.
\begin{proof}
It suffices to consider the map 
\[
	F(x) := \left(\begin{array}{c} E(x) \\  \\ M(x) \end{array}\right)
\]
and show the differential is surjective for some $x \in \r^3$. But this differential is
\[
	DF(x) = \left(\begin{array}{ccc} 2\lambda_1 \, x_1 & 2\lambda_2 \, x_2 & 2\lambda_3 x_3 \\ & & \\ 2\lambda_1^2 \, x_1 & 2\lambda_2^2 \, x_2 & 2\lambda_3^2 \, x_3 \end{array} \right),
\]
whose minor determinants are (up to a factor of $\pm2$) 
\[
\begin{array}{l}
	\lambda_1 \, \lambda_2 \, (\lambda_2 - \lambda_1) \, x_1 \, x_2,\\
	\\
	\lambda_1 \, \lambda_3 \, (\lambda_3 - \lambda_1) \, x_1 \, x_3, \; \text{and}\\
	\\
	\lambda_2 \, \lambda_3 \, (\lambda_3 - \lambda_2) \, x_2 \, x_3.
\end{array}
\]
Since we are assuming that not all $\lambda_i$ are equal, then choosing all $x_i$ non-zero will give maximal rank. 
\end{proof}
\vs
Let us next study the curve $C_{a,b}$ where we assume $a, b$ are generic. 
\vs
\begin{lemma} \label{lem:lambdadistinct}
	Suppose the eigenvalues $\lambda_i$ of $A$ are pairwise distinct, and $a, b$ are generic, as in lemma \ref{lem:gpab}. Then the projective curve $C_{a,b}$ is smooth, of genus 1.
\end{lemma}
\begin{proof} Let $B$ be the two by four matrix
\[
	B = \left(\begin{array}{cccc} \lambda_1 z_1 & \lambda_2 \, z_2 & \lambda_3 \, z_3 & - a \, w \\	& & & \\ \lambda_1^2 \, z_1 & \lambda_2^2 \, z_2 & \lambda_3^2 \, z_3 & - b \, w \end{array} \right).
\]  
We wish to show that there are no $(z_1, z_2, z_3, w) \neq (0,0,0,0)$ which solve (\ref{eqn:qe}) and (\ref{eqn:qm}), and for which the rank of $B$ is $<$ 2, or the six $2\times 2$ minors of $B$ are 0:
\[
\begin{array}{rcl} 
	0 & = & \lambda_1 \, \lambda_2 (\lambda_2 - \lambda_1) \, z_1 \, z_2 \\
	\\
	0 & = & \lambda_1 \, \lambda_3 \,(\lambda_3 - \lambda_1)\, z_1\, z_3 \\
	\\
	0 & = & \lambda_2 \, \lambda_3 \, (\lambda_3 - \lambda_2)\, z_2 \, z_3 \\
	\\
	0 & = & \lambda_1\, (b - \lambda_1 \, a) \, z_1 \, w\\
	\\
	0 & = & \lambda_2 \, (b - \lambda_2 \, a) \, z_2 \, w \\
	\\
	0 & = & \lambda_3 \, (b - \lambda_3 \, a) \, z_3 \, w.
\end{array}
\]
On the other hand, it is clear that any 4-tuple $(z_1, z_2, z_3, w)$ satisfying (\ref{eqn:qe}) or (\ref{eqn:qm}) must have at least two non-zero components, while on the other hand, all the constants in this last array of equations are non-zero under our assumptions. This is a contradiction, hence $C_{a,b}$ is smooth.
\vs
To complete the proof, we use the classical genus formula
\[
	g(C_{a,b}) = 1 + \frac12 \; C_{a,b} \cdot (C_{a,b} + K_{\q_{E,a}}),
\]
where $K_{\q_{E,a}}$ is the canonical class of $\q_{E,a} \sim -2 H$, where $H$ in turn is the hyperplane section of $\q_{E,a} \subset \p^3$. Since $C_{a,b} \sim 2 H$, we see that $g(C) = 1$, as claimed.
\vs
\end{proof}
Now the map $v: \c \to C_{a,b} \subset \p^3$ associated to a complexified geodesic with initial conditions $a = E(\gamma(0)), b = M(\gamma(0)),$ must actually map into the affine piece $C^{{aff}}_{a,b} = C_{a,b} \cap \mfg_{\c} = C_{a,b} \setminus C_{a,b} \cap \{w = 0\}$. But this last intersection is non-empty and a holomorphic map of $\c$ to $C^{aff}_{a,b}$ must factor through the universal covering $\pi: \mathbb{C} \to C_{a,b}$, but must omit the infinite set $\pi^{-1}(C_{a,b} \cap\{w = 0\})$. Thus, the map $c$ must be constant, by Picard's theorem. In real terms, this means that the real geodesic satisfies
\[
	\dot{\gamma} \equiv \xi \in \mfg,
\]
If we assume, without loss of generality, that $\gamma(0) = e \in G,$ this means that $\gamma(s)$ is the 1-parameter subgroup $\exp(s \xi)$ of $G$. Now this conclusion holds true for all initial conditions $\dot{\gamma}(0) = \xi_0$ in a dense open set of $\mfg = T_0G.$ By continuity of solution of ODEs, this must hold true for all geodesics passing through 0, and hence, by left translation, for all geodesics. However, it is well-known that any left invariant metric for which all geodesics through $e$ are 1-parameter subgroups is necessarily bi-invariant and hence round (see e.g. \cite{o}, ch. 11, Proposition 9). This would contradict our assumption that the $\lambda_i$ are all distinct. 
\vs
{\em Remark:} One can give a more elementary proof of this last statement using (\ref{eqn:geodeqns}) in the appendix \cite{web} to show that, if the $\lambda$'s are distinct, $\gamma(t)$ is a geodesic with $\gamma(0) = e \in G$, then $T = \dot{\gamma}$ is constant in the left invariant frame if and only if $\gamma(t)$ is a geodesic $\exp(t \, \nu \, \xi_i), i = 1, 2,$ or $3$, for some real constant $\nu$.
\vs
Next, consider the case where two of the $\lambda_i$ are equal. This is equivalent to the fact that the identity component of the isometry group of $g_A$ is $G\times S^1$ where the extra circle group is a 1-parameter subgroup of $G$ operating on the right. To be specific, we set $\lambda_1 = \lambda_2$, the other case being similar. In this case, equations (\ref{eqn:qe}) and (\ref{eqn:qm}) can be reduced to the pair of equations (\ref{eqn:qe}) and
\begin{equation} \label{eqn:split}
	\begin{array}{rcl} 0 & = & (\lambda_1 \lambda_3 - \lambda_3^2) \, z_3^2 - (\lambda_1 \, a - b) \, w^2 \\
					& & \\
					& = & \lambda_3 \, (\lambda_1 - \lambda_3) \, ( z_3^2 - \frac{\lambda_1 \, a - b}{\lambda_3 \, (\lambda_1 - \lambda_3)} \, w^2)
	\end{array}
\end{equation}
Using our notation for the initial condition $\dot{\gamma}(0) = \xi_0 = x_1^o \, \xi_1 + x_2^o \, \xi_2 + x_3^o \, \xi_3,$ we can express
\[
	a = E(x^o), b = M(x^o),
\]
from which we can solve for 
\[
	\frac{\lambda_1 \, a - b}{\lambda_3 \, (\lambda_1 - \lambda_3)} = (x_3^o)^2.
\]	
It follows that (\ref{eqn:split}) splits over $\r$ and the curve defined by (\ref{eqn:qe}) and (\ref{eqn:split}) is the union of two curves
\[
	C^+_{a,b} = \q_{E,a} \cap \{ z_3 = x_3^o \, w\} 
\]
and
\[	
	C^-_{a,b} = \q_{E,a} \cap \{z_3 = - x_3^o \, w\}. 
\]
Each of these curves is a conic $\p^1 \subset \p^3$ which intersects the plane at infinity $\{w = 0\}$ in two points generically, so $C^{\text{aff}, \pm}_{a,b}$ is generically biholomorphic to $\c^*$, which certainly contains the images of non-trivial entire mappings from $\c$. Thus the method of proof for the case of $\lambda_i$ pairwise distinct will not work here. In fact, if $\lambda_1 = \lambda_2$ then the tube of the metric $g_A$ is sometimes entire.
%
%
%
%
%
%
%
%
%
%
%
\section{The cases $\lambda_1 = \lambda_2$} \label{sec:U2cases}
\vs
By scaling and renumbering, we can reduce to the case 
\begin{equation}\label{eqn:U2sym}
	g_A(v,v) = v_1^2 + v_2^2 + \lambda v_3^2,
\end{equation}
where $\lambda > 0$. The behavior of the associated tube is different according to whether $\lambda \leq 1,$ or $\lambda > 1$. We treat first the case where the tube is entire.
\vs
\begin{proposition} \label{prop:poscase}
	Let $g_A$ be as in (\ref{eqn:U2sym}) with $\lambda \leq 1$. Then the Grauert tube of $g_A$ is entire.
\end{proposition}
\begin{proof} This is an immediate consequence of a result of Sz\H{o}ke \cite{s2} (cf also \cite{a1}). The metric (\ref{eqn:U2sym}) has an isometry group $G \times S^1$ where the factor $S^1$ acts by right multiplication on $G$ by the 1-parameter subgroup $e^{t\, \xi_3}$ on the right, i.e., by $R_{e^{t\, \xi_3}}$, periodic of period $2\pi$, and the factor $G$ acts by left multiplication $L_g$ on $G$. Consider the manifold $G \times S^1$, where $G$ is given the bi-invariant metric of sectional curvature 1, which is the special case of (\ref{eqn:U2sym}) with $\lambda_1= 1$, and the $S^1$ factor is given the metric $\mu d\theta^2,$ where $d\theta^2$ is the invariant metric of length $2\pi$ on $S^1$ and $\mu$ is a real constant $> 0$. So $G\times S^1$ has the metric $g_{\mu}$ defined by
\begin{equation}  \label{eqn:gmu}
	g_{\mu} = g_I \oplus \mu \, d\theta^2.
\end{equation}

Consider the quotient manifold of $G\times S^1$ by the action of $S^1 = \{e^{i t}\}$ on $G\times S^1$ given by $e^{i t} \cdot (g, e^{i \theta}) = (g\cdot e^{2 t\, \xi_3}, e^{i(\theta + t)}).$ Then it is easy to see that the quotient manifold $G\times S^1/\sim$ is diffeomorphic to $G$. Let $\bar{g}_{\mu}$ be the Riemannian submersion metric from (\ref{eqn:gmu}). By construction, this metric is invariant under left translation by $G$ and by right translation by $e^{- t \, \xi_3}$. Consequently, it is one of the left invariant metrics on  $G$ with an extra isometric action from the right. It is an easy calculation to see which ones arise in this manner. Let 
\begin{equation} 
	\label{eqn:A}
		A = \left(\begin{array}{ccc} 1 & & \\ & 1 & \\ & & \lambda \end{array}\right).
\end{equation}
Explicitly, we will normalize our metrics so that $g_A(\xi_1, \xi_1) = g_A(\xi_2, \xi_2) = \frac14,$ and $g_A(\xi_3, \xi_3) = \frac{\lambda}4$.\footnote{This means that instead of $g_A(\cdot,\cdot) = -B(A(\cdot), \cdot)$, we now have $g_A(\cdot, \cdot) = -\frac18 B(A(\cdot), \cdot)$.} With this normalization, the case $A = I$ gives the round metric with sectional curvature $K = 1$ on $G \cong S^3$.
\vs
We want to see which $\lambda$ in (\ref{eqn:U2sym}) will arise from Sz\H{o}ke's construction. Computing at $(e,1) \in G\times S^1$, the tangent vector to the $S^1$ action we will quotient by is given by $(2\xi_3, \frac{\partial \;\;}{\partial \theta})$. The projection of $(\xi_3, 0)$ onto this direction is given by
\begin{equation} \label{eqn:Tproj}
	\begin{array}{rcl}
	(\xi_3,0)^T & = & \frac{((\xi_3, 0), (\xi_3, \frac{\partial \;\;}{\partial \theta}))}{|\xi_3 + \frac{\partial \;\;}{\partial \theta}|^2} (\xi_3, \frac{\partial \;\;}{\partial \theta}), \\
	& & \\
	& = & \frac2{2+\mu} (\xi_3, \dbdth)
	\end{array}
\end{equation}
and hence the normal component is given by 
\begin{equation} 
	\label{eqn:Nproj}
	\begin{array}{rcl}
		(\xi_3,0)^N & = & (\xi_3,0) - (\xi_3,0)^T \\
		&	&	\\
		& = & (\frac{\mu}{2+\mu} \xi_3, \frac2{2+\mu} \dbdth). 	
	\end{array}
\end{equation}
Thus, 
\[
	|\xi_3^N|^2 = 2 \frac{\mu}{2+\mu},
\]
which means that the quotient metric on $G$ is $g_A$ with $A$ in (\ref{eqn:A}) having $\lambda = \frac{\mu}{2 + \mu}.$ Hence, any $\lambda \in (0, 1)$ appears. The limit case $\mu = +\infty$ is just the case of the round sphere, $\lambda = 1$, which also has an entire tube, of course.
\end{proof}
\vs
Finally, there remains the case of (\ref{eqn:A}) with $\lambda > 1$. These metrics have tubes of finite radius, and we study the complexified Jacobi equation to determine this radius. We first observe that, in contrast to the case of a generic $A$, the map $Exp_{g_{A,\c}}$ is globally defined for real initial conditions, i.e., on $TG$.
\begin{lemma}\label{lem:Expdefd}
	For a metric $g_A$ with $G\times S^1$ symmetry (\ref{eqn:A}), the $g_{A,\c}$ exponential map is defined and smooth on all of $TG$.	
\end{lemma}
\begin{proof} 
	Let $\gamma$ be the real geodesic on $G$ with initial conditions $\gamma(0) = e \in G$ and $\dot{\gamma}(0) = v_0 \in T_eG \cong \mfg$. We solve the geodesic equations (\ref{eqn:geodeqns1}) in the appendix \cite{web} relatively explicitly to show that the complexified geodesic $\gamma_{\c}(\zeta)$ is an entire function from $\c$ to $G_{\c}$. The other statements in the proposition follow from smooth dependence on initial conditions.
\vs
For convenience, we repeat here the equations (\ref{eqn:geodeqns1}) from \cite{web} for the tangent vector $T = a\, \xi_1+b\, \xi_2+c\, \xi_3$ to a geodesic, subject to the initial condition $T(0) = a_0 \, \xi_1+ b_0\, \xi_2 +c_0 \, \xi_3,$ for a metric of the form $g_A$,  with $A$ as in (\ref{eqn:A}):
\[
	\begin{array}{rcl}
		\dot{a} & = & -b\, c\, (\lambda - 1) \\
		&&\\
		\dot{b} & = & - a\, c\, (1 - \lambda) \\
		&&\\
		\dot{c} & = & 0,
	\end{array}
\]
where $\dot{a} = T(a)$, etc. We get that $c \equiv c_0 \in \r$, and the system reduces to 
\[
	\begin{array}{rcr}
		\dot{a} & = & -\nu\, b \\
		&&\\
		\dot{b} & = & \nu \, a, 
	\end{array}
\]
with $\nu = c_0 \, (\lambda - 1) \in \r$. Hence, 
\vs
\[
	\begin{array}{rcl} a & \equiv & a_0 \cos \nu \, t - b_0 \sin \nu \,t \\&&\\ b & \equiv & b_0 \cos \nu \,t + a_0 \sin \nu \, t, \end{array}
\]
\vs
where $t$ is the parameter of integration in the equations. Set 
\[
	T_{\c} = a(\zeta) \, \xi_{1,\c} +b(\zeta) \, \xi_{2,\c} + c_0 \, \xi_{3,\c},
\] 
which is an entire function from $\c \to \mathfrak{sl}(2,\c) \subset \c^4.$ Considering $\gamma_{\c}$ as a function from, at first, a neighborhood of $\r \subset \c$ to $G_{\c} = SL(2,\c) \subset \c^4$, we have that it satisfies the equation 
\[
	\gamma_{\c}(\zeta)^{-1} \cdot \dot{\gamma_{\c}} = T_{\c},
\]
where $\dot{\gamma}_{\c} = \frac{d\;\;}{d \zeta}$, since we are using the left-invariant parallelism to trivialize $T^{1,0}G_{\c}$. Equivalently, we have
\[
	  \dot{\gamma}_{\c} = \gamma_{\c} \cdot T_{\c},
	 \]
which is a linear equation, with $T_{\c}$ entire and known, on all of $\c$. Hence all its solutions extend to entire functions, proving the lemma.
\end{proof}
\vs
Because of lemma \ref{lem:Expdefd}, the map $\Phi$ of section \ref{sec:Gcgeom} is defined and smooth on all of $TG$. We next show that for $\lambda> 1$ this map has a singularity at 
$t_*(\lambda) \in \r_+.$
\vs
\begin{proposition}\label{prop:focus}
	Let $A$ be as in (\ref{eqn:A}) with $\lambda> 1$. Then $D\Phi$ is singular at $(e, t_*(\lambda) \xi_1) \in T_eG \cong \mfg,$ for some $t_*(\lambda) \in (0,\infty)$.	
\end{proposition}
\begin{proof}
	By lemma \ref{lem:rankDPhi}, to find singularities of $D\Phi$ at points $t v \in T_eG$, it suffices to \vs
	\begin{enumerate}
		\item Solve the Jacobi equation $\mathcal{J}_1(Y) = 0, Y = f \,\xi_2 + h\, \xi_3,$ along $\gamma$, the geodesic in $G$ with initial condition $\gamma(0) = e \in G$, and $\dot{\gamma}(0) = \xi_1 \in T_eG \cong \mfg$.\vs\vs
		\item Set $Y_i = f_i\, \xi_2 + h_i \, \xi_3, i =2, 3,$ where $f_i(0) = h_i(0) = 0,$ and \vs 
			\begin{equation}\label{eqn:Yinit}
				\left(\begin{array}{cc} f_2'(0)& h_2'(0) \\& \\f_3'(0) & h_3'(0) \end{array}\right) = \left(\begin{array}{cc} 1 & 0 \\& \\ 0 & 1\end{array}\right).
			\end{equation}\vs\vs
		\item Express $Y_{i,\c} = \alpha_i\, \eta_{2,\c}^{1,0} + \beta_i \, \eta_{3,\c}^{1,0}, i = 2,3.$\vs\vs
		\item Evaluate \[\Delta(t) = \det Im \left(\begin{array}{cc} \alpha_2(i t) & \beta_2(i t) \\ & \\ \alpha_3(i t) & \beta_3(i t) \end{array} \right).\] Solve $\Delta(t) = 0$.
	\end{enumerate}\vs\vs
We will show that $\Delta(t) = 0$ has a unique root $t_*(\lambda) \in (0, +\infty)$, proving the proposition.
\vs\vs
Step (1) is calculated in section \ref{sub:JacA}, (\ref{eqn:Jac1soln}):
	\[
	\begin{array}{rcl}
		f & = & a \cos t + b \sin t + c\\&&\\h & = &  b \cos t - a \sin t + \frac{1-\lambda}{\lambda} c t + d, 
	\end{array}
	\]
where $a, b, c ,d \in \r$. (2) says, first, that $c = - a, d = - b,$ and further the first order initial conditions say
\vs
\[
\begin{array}{rcccl} f'_2(0) & = & b & =& 1 \\&&&&\\h'_2(0) & =& - \frac{a}{\lambda} & = & 0\end{array}
\]
\vs
and
\vs
	\[
	\begin{array}{rcccl}f'_3(0) & = & b & = & 0\\ &&&&\\ h'_3(0) & = & - \frac{a}{\lambda} & =& 1.\end{array}
	\]
	\vs
Hence, in the basis $\xi_2, \xi_3$ for normal fields to $\gamma$, we have
\vs
\begin{equation}\label{eqn:fhsols}
		\left(\begin{array}{cc} f_2 & f_3 \\&\\ h_2 & h_3 \end{array}\right) = \left(\begin{array}{cc} \sin t & - \lambda (\cos t - 1) \\&\\  \cos t - 1& \lambda \sin t + (1 - \lambda) t \end{array}\right).
\end{equation}
\vs
As to step (3), as  in $(\ref{eqn:xi2eta})$ we have at $g \in \exp t \xi_1$ the change of basis from the $\xi$'s to the $\eta$'s is given by the action $Ad(- t \xi_1)$, explicitly:
\vs
\[
	\begin{array}{rcccl}
		\eta_2 & = & Ad(\exp(- t \xi_1)) \, \xi_2 & = & \cos t \; \xi_2 + \sin t \; \xi_3, \\
			&&&&\\
		\eta_3 & = & Ad(\exp(- t \xi_1)) \xi_3 & = & - \sin t \; \xi_2 + \cos t \; \xi_3.
	\end{array}
\]
\vs
so that we have 
\vs
\[
	\begin{array}{rcl}
		\left(\begin{array}{cc} \alpha_2 & \alpha_3\\ & \\ \beta_2 & \beta_3 \end{array}\right) & = & \left( \begin{array}{cc} \cos t & - \sin t \\& \\ \sin t & \cos t \end{array}\right) \cdot  \left(\begin{array}{cc} \sin t & - \lambda (\cos t - 1) \\&\\  \cos t - 1& \lambda \sin t + (1 - \lambda) t \end{array}\right) \\ && \\
		& = & \left(\begin{array}{cc} \sin  t & \lambda [(\cos t - 1) + \frac{\lambda - 1}{\lambda} \, t \sin t] \\ & \\ - (\cos t -1) & \lambda [\sin t - \frac{\lambda - 1}{\lambda} \, t \cos t] \end{array}\right). 
	\end{array}
\]
\vs
It follows directly that
\vs
\[ 
	Im \left(\begin{array}{cc} \alpha_2(i t) & \alpha_3(i t) \\ & \\ \beta_2(i t) & \beta_3(i t)\end{array}\right) = \left(\begin{array}{cc} \sinh t & 0 \\ & \\ 0 & \lambda [\sinh t - \frac{\lambda - 1}{\lambda} \, t \cosh t] \end{array}\right).
\]
\vs
Finally, we have $\Delta(t) = 0$ if and only if 
\vs
\begin{equation}\label{eqn:tstar}
	\frac{\lambda - 1}{\lambda} \, t = \tanh t.
\end{equation}
\vs
It is clear pictorially or from an elementary calculus computation that, if $\lambda >1$, there is a unique $t_*(\lambda) > 0$ satisfying (\ref{eqn:tstar}), as claimed
\end{proof}
%
%
%
%
%
%
%
%
%
%
%
%
\section{Remarks and Questions}\label{sec:rmks}
There are a couple of questions related to the present paper, beyond the general one of the construction and interpretation of geometries with entire tubes. Can one carry out such a study as here for groups other than $SU(2)$? It would seem to require integrability with especially simple integrals, such as the quadratic polynomials here. It may also be possible to determine the finite radius of the tubes for $\lambda \in (1,+\infty)$ using the explicit solution for the tangent vector along a geodesic for such geometries, as in section \ref{sec:U2cases} above. One might also find an estimate for the tube radius using an approximation to the constant negative curvature metrics on the affine elliptic curves given by the complex integrals of motion and the Ahlfors-Schwarz-Pick inequality. It is unclear how such bounds would compare to the Lempert-Sz\H{o}ke curvature bound when both bounds apply. 
\vs
Finally, can anything interesting be said of the case of a compact Riemannian homogeneous space $G/H$ which is not normal, i.e., not the Riemannian quotient of a symmetric metric on $G$? A related question would be whether Gelfand-Cetlin quantum integrable systems exist on flag manifolds, which contain the Laplacian for a non-normal homogeneous metric. 
%
%
%
%
%
%
%
%
%
%
%
%
%
%
%
%
%
%
%
%

%
%
%
%
%
%
%
%
%
%
%
%
%
%
%
%
\section{Appendix: Riemannian geometry details for $SU(2)$} \label{sec:app}
%
%
We insert here for easy reference several well-known and standard computations useful in the paper, first for the left invariant geometries on $G$ and then for their complexifications on $G_{\c}$. We will compute the Levi-Civita connection of $g_A$ on $G = SU(2)$, the geodesic equations, the Riemann curvature tensor and the Jacobi equations, the latter two for the more symmetric metrics as in (\ref{eqn:A}).
%
%
%
%
%
%
%
\subsection{The Levi-Civita connection of $g_A$ on $G = SU(2)$} \label{sub:LCA} We let $(\cdot,\cdot) = -B(\cdot,\cdot)$ be the invariant form on $\mfg$, and we view it as a metric on all of $G$ by left-invariance. We will compute the Levi-Civita connection of $g_A$ on $G = SU(2)$, the geodesic equations. Since all of these objects will be left-invariant, they will amount to multilinear forms on $\mfg$, by computing them with respect to a left-invariant frame of $TG$, equivalently a basis for $\mfg$. The case of $\mfg = \mathfrak{su}(2)$ is particularly simple in this regard, since we can find an orthonormal basis for $(\cdot,\cdot)$ consisting of $\xi_i, i = 1,2,3,$ with brackets given by
\begin{equation} \label{eqn:structureconstants}
	[\xi_i, \xi_{i+1}] = \xi_{i+2},
\end{equation}
where the indices are read$\mod 3$, i.e., permuted cyclically. In fact, if we are given $A$ as above, we can assume this basis also diagonalizes $A$ on $\mfg$:
\begin{equation} \label{eqn:eigenbasis}
	A(\xi_i) = \lambda_i \, \xi_i, i = 1,2,3, \,\text{with} \, 0 < \lambda_1 \leq \lambda_2 \leq \lambda_3,
\end{equation}
Thinking of $\xi, \eta \in \mfg$ as left-invariant fields on $G$, we want to compute the Levi-Civita connection $\naba$ for the metric $g_A = (\cdot,\cdot)_A = (A\cdot,\cdot)$ applied to $\xi, \eta$, that is, $\nabla_{\xi}\, \eta.$ Let us first recall the standard formula for the Levi-Civita connection for an arbitrary Riemannian manifold, where momentarily, $(\cdot,\cdot)=g$ is arbitrary and $\nabla$ its Levi-Civita connection. Let $X, Y, Z$ be arbitrary smooth vectorfields, then $\nabla$ is determined by:
\begin{equation} \label{eqn:LC}
	\begin{array}{rcl}
		(\nabla_X\, Y, Z) & = & \frac12 \{X(Y, Z) - Z(X, Y) + Y(Z, X) +\\
					&	&	\\
					&	&  (Z, [X, Y]) + ([Z, X], Y) + (X, [Z, Y]) \}. 
	\end{array}
\end{equation}
Specializing now to our case of a left-invariant metric on $G$, we get
\begin{equation} \label{eqn:LCG}
	(\naba_{\xi}\, \eta, \zeta)_A = \frac12\{ (\zeta, [\xi, \eta])_A + ([\zeta, \xi], \eta)_A  + (\xi, [\zeta, \eta])_A\},
\end{equation}
where $\xi, \eta, \zeta$ are left-invariant fields in $\mfg$. In particular, if $G = SU(2)$, we can use the basis $\xi_1, \xi_2, \xi_3$ from (\ref{eqn:structureconstants}) and (\ref{eqn:eigenbasis}) to get explicitly:
\begin{equation} \label{eqn:LCG1}
	\naba_{\xi_j}\, \xi_i = \sum_k a_{i; j}^k \xi_k,
\end{equation}
where the $a_{i; j}^k$ are real constants given by
\begin{equation} \label{eqn:1pggeod}
	\begin{array}{rcl}
		0 	& = & \naba_{\xi^1}\xi_1 = \naba_{\xi_2}\xi_2 = \naba_{\xi_3}\xi_3 , 
	\end{array}
\end{equation}
and		 
\begin{equation} \label{eqn:covder}
	\begin{array}{rcl}
		\naba_{\xi_1}\xi_2 & = & \frac1{2\lambda_3}\{-\lambda_1 + \lambda_2 + \lambda_3\} \, \xi_3 \\
			&	&	\\
		\naba_{\xi_1}\xi_3 & = & \frac1{2 \lambda_2} \{\lambda_1 - \lambda_2 - \lambda_3\}\, \xi_2 \\
			&	&	\\
		\naba_{\xi_2} \xi_1 & = & \, \frac1{2 \lambda_3} \{-\lambda_1 +\lambda_2 - \lambda_3\} \, \xi_3 \\
			&	&	\\
		\naba_{\xi_2} \xi_3 & = & \frac1{2 \lambda_1} \{\lambda_1 - \lambda_2 + \lambda_3\} \, \xi_1\\
			&	&	\\
		\naba_{\xi_3} \xi_1 & = & \frac1{2 \lambda_2}\{\lambda_1 + \lambda_2 - \lambda_3\} \, \xi_2 \\
			&	&	\\
		\naba_{\xi_3} \xi_2 & = & \, \frac1{2 \lambda_1} \{-\lambda_1 - \lambda_2 + \lambda_3\} \, \xi_1
	\end{array}
\end{equation} 
%
%
%
%
%
%
\subsection{The geodesic equations.} We next write out the geodesic equations for the tangent vector $T$ to a geodesic, $\nabla_TT = 0$ on the real space $SU(2)$. 
Write $T = a \xi_1 + b \xi_2 + c \xi_3$. We shorten $\nabla^A$ to simply $\nabla$.Then
\[\begin{array}{cclccccc}
	\nabla_TT & = &  & \dot{a} \xi_1 & + & \dot{b} \xi_2 & + & \dot{c} \xi_3 \\
	&&&&&&&\\
	&& + & \dot{a} \xi_1 & + & \dot{b} \xi_2 & + & \dot{c} \xi_3 \\
	&&&&&&&\\
	&& + & \dot{a} \xi_1 & + & \dot{b} \xi_2 & + & \dot{c} \xi_3 \\
	&&&&&&&\\
	&& + & a^2 \nabla_{\xi_1}\xi_1  & + & ab \nabla_{\xi_1}\xi_2  & + & ac \nabla_{\xi_1}\xi_3 \\
	&&&&&&&\\
	&& + & ba \nabla_{\xi_2}\xi_1 & + & b^2 \nabla_{\xi_2}\xi_2 & + & bc \nabla_{\xi_2}\xi_3 \\
	&&&&&&&\\
	&& + & ca \nabla_{\xi_3}\xi_1 & + & cb \nabla_{\xi_3}\xi_2 & + & c^2 \nabla_{\xi_3}\xi_3,	
	\end{array}\]
where $\dot{a} = T(a)$, etc. Use (\ref{eqn:1pggeod}) and (\ref{eqn:covder}) to reduce these to
\begin{equation}\label{eqn:geodeqns}
	\begin{array}{lcr}
		\dot{a} & = & - \frac{bc}{\lambda_1}(\lambda_3 - \lambda_2),\\
		&&\\
		\dot{b} & = & - \frac{ac}{\lambda_2}(\lambda_1 - \lambda_3),\\
		&&\\
		\dot{c} & = & - \frac{ab}{\lambda_3}(\lambda_2 - \lambda_1).
	\end{array}
\end{equation}
A special case of interest is when $\lambda_1 = \lambda_2 = 1, \lambda_3 := \lambda,$ where the equations reduce to 
\begin{equation}\label{eqn:geodeqns1}
	\begin{array}{rcl}
		\dot{a} & = & -bc(\lambda - 1) \\
		&&\\
		\dot{b} & = & - ac(1 - \lambda) \\
		&&\\
		\dot{c} & = & 0.
	\end{array}
\end{equation}
%
%
%
%
%
%
%
%
%
%
\subsection{Riemann curvature of $g_A$.} \label{sub:RCA} Next, we describe the curvature tensor of $g_A$, where $A$ is as in (\ref{eqn:A}), and for $A = I, |\xi_i|^2 = \frac14.$ Thus, $|\xi_i|^2_{\lambda} = \frac14, i = 1,2,$ and $|\xi_3|^2_{\lambda} = \frac{\lambda}4.$  We view the Riemann curvature tensor as a symmetric operator also denoted $Riem: \wedge^2\mathfrak{g} \to \wedge^2 \mfg$. From the additional $S^1$ invariance of our metric, we conclude that the eigenspaces of $K$ are $V_m := \{ \r \xi_1 \wedge \xi_2\}$ and $V_M = \{ \r \xi_1 \wedge \xi_3 + \r \xi_2 \wedge \xi_3\}$, and where the eigenvalues are given by the sectional curvatures $K(\xi_1\wedge\xi_2)$, $K(\xi_1, \xi_3) = K(\xi_2, \xi_3)$, respectively. It follows that $Riem(\xi_i, \xi_j) \xi_k = 0,$ if $i, j, k$ are distinct. One calculates the sectional curvatures as follows:
\begin{equation} \label{eqn:secK}
	K(\xi_i \wedge \xi_j) = \frac{(Riem(\xi_i, \xi_j) \xi_j, \xi_i)}{|\xi_i \wedge \xi_j|^2}.
\end{equation}
For the case of $\xi_1 \wedge \xi_2$, we have 
\[
	\begin{array}{rcl}
		\nabla_{\xi_2} \circ \nabla_{\xi_1} \xi_1 & = & 0, \\
			&	&	\\
		\nabla_{\xi_1} \circ \nabla_{\xi_2} \xi_1 & = & - \nabla_{\xi_1} \frac12 \xi_3 \\
		&	&	\\
			& = & \frac{\lambda}4 \xi_2, \; \text{and} \\
			&	&	\\
		\nabla_{[\xi_2, \xi_1]} \xi_1 & = & - \nabla_{\xi_3} \xi_2 \\
		& & \\
			& = & (\frac{\lambda}2 - 1)	 \xi_2,
	\end{array}
\]
or,
\begin{equation} \label{eqn:R12}
	Riem(\xi_2, \xi_1) \xi_1 = (1 - \frac34 \lambda) \xi_2.
\end{equation}
\vs
and therefore
\begin{equation} \label{eqn:K12}
	K(\xi_1 \wedge \xi_2) = \frac{(Riem(\xi_2, \xi_1) \xi_1, \xi_2)}{|\xi_1\wedge \xi_2|^2} = 16 \cdot (1- 3 \frac{\lambda}4) \cdot \frac14 = 4 - 3 \lambda.
\end{equation}
Similar calculations show that
\begin{equation}\label{eqn:K31}
	R(\xi_3, \xi_1) \xi_1 = \frac{\lambda}4 \xi_3.
\end{equation}
\begin{equation} \label{eqn:R13}
	Riem(\xi_1, \xi_3) \xi_3 = \frac{\lambda^2}4 \xi_1.
\end{equation}
\begin{equation}\label{eqn:R23}
		Riem(\xi_2, \xi_3) \xi_3 = \frac{\lambda^2}4 \xi_2.
\end{equation}
As a result, we get
\begin{equation} \label{eqn:K13}
	K(\xi_1 \wedge \xi_3) = K(\xi_2 \wedge \xi_3) = \frac{\frac{\lambda^2}4}{\frac{\lambda}4} = \lambda.
\end{equation}
We conclude that the sectional curvatures of $S^3, g_A$ lie in $[\lambda, 4 - 3 \lambda]$, for $\lambda \in (0, 1]$, and they lie in $[4 - 3\lambda, \lambda]$, for $\lambda \geq 1$. Thus the sectional curvatures for $g_A$ are positive for $\lambda \leq \frac43$.
\vs
%
%
%
%
%
%
%
\subsection{The Jacobi operator for $g_A$} \label{sub:JacA} In general, if $\gamma$ is a Riemannian geodesic with $\dot{\gamma} := T$, so that $\nabla_T T = 0$, the Jacobi operator $\mathcal{J}$ associates to a vectorfield $Y$ along $\gamma$ and normal to $\gamma$ the normal field along $\gamma$ given by
\[
	\mathcal{J}(Y) := \nabla_T^2 Y + Riem(Y, T) T.
\]
We will calculate this operator when $T = \xi_1$. 
%
First calculate $\mathcal{J}(\xi_2)$ along $\gamma$ with $\dot{\gamma} = T = \xi_1$.
Call this particular Jacobi operator $\mathcal{J}_1$. 
We have, by equations (\ref{eqn:covder}, \ref{eqn:R12}), 
\[
	\begin{array}{rcl}
		\mathcal{J}_1(\xi_2) & = & \nabla_{\xi_1}^2 \, \xi_2 + Riem(\xi_2, \xi_1) \, \xi_1 \\
			& & \\
			& = & \nabla_{\xi_1} \, \frac12 \xi_3 + (1 -\frac34 \, \lambda) \, \xi_2 \\
			&	&	\\
			& = & - \frac{\lambda}4 \xi_2\, + (1- \frac34 \, \lambda) \xi_2 \\
			&	&	\\
			& = & (1 - \lambda) \xi_2
	\end{array}
\]
and so
\begin{equation} \label{eqn:Jac12}
	\mathcal{J}_1(\xi_2) =  (1 - \lambda) \, \xi_2,
\end{equation}
while
\begin{equation}\label{eqn:Jac13}
	\mathcal{J}_1(\xi_3) = 0.
\end{equation}
since $\xi_3$, being tangent to the $S^1$ symmetry, is a Jacobi field along $\gamma$.
\vs\vs
A normal Jacobi field $Y$ along $\gamma$ is a field along $\gamma$ which may be written $Y =f\, \xi_2 + h \,\xi_3$ which satisfies $\mathcal{J}_1(Y) = 0$. This gives us a $2 \times 2$ system of coupled ODEs, using (\ref{eqn:Jac12}, \ref{eqn:Jac13}), as well as (\ref{eqn:covder}) :
\vs
\begin{equation} \label{eqn:Jac1sys}
	\begin{array}{rcl}
		f'' - \lambda h' + (1 - \lambda) \, f & = & 0 \\
			&	&	\\
		h'' + f'  & = & 0.
	\end{array}
\end{equation}
\vs
Differentiating the first equation of (\ref{eqn:Jac1sys}), setting $F := f'$, and substituting for $h''$ using the second equation of (\ref{eqn:Jac1sys}), we get
\vs
\[
	F'' + F = 0,
\]
\vs
so 
\begin{equation}\label{eqn:f}
	f =a \cos t + b \sin t + c,
\end{equation}
\vs
with $a, b, c \in \r$. The second equation in (\ref{eqn:Jac1sys}) then says 
\vs
\[
	h = - a \cos t + b \sin t + \tilde{c}\; t + d, \;\; \tilde{c}, d \in \r.
\]
\vs
Now use the first equation in (\ref{eqn:Jac1sys}) to solve for $\tilde{c} = \frac{1-\lambda}{\lambda} \, c$, to obtain, finally,
\vs
\begin{equation}\label{eqn:Jac1soln}
	\begin{array}{rcl} f & = & a \cos t + b \sin t + c, \;\; \text{and} \\ && \\ h & = & b \cos t - a \sin t + \frac{1- \lambda}{\lambda} \, c \, t + d. \end{array}
\end{equation}
\vs\vs\vs

\begin{thebibliography}{99}

\bibitem{a1} R. Aguilar, Symplectic reduction and the homogeneous Monge-Amp\`ere equation, \emph{Ann. Global Anal. and Geom.} {\bf 19} (2001), 327-353.

\bibitem{a2} R. Aguilar, Higher order Schwarzians for geodesic flows, moment sequences, and the radius of adapted complexifications, \emph{Q. J. Math.} {\bf 64} (2013), 1-36.

\bibitem{web} V. Aslam, D. Burns, D. Irvine, Left-invariant Grauert tubes on SU(2), section 7 (appendix), \emph{arXiv:1705.03359}. 

\bibitem{via} V. I. Arnol'd, Mathematical Methods of Classical Mechanics, Springer Verlag, New York, 1978.

\bibitem{gs1} V.W. Guillemin and M. Stenzel, Grauert tubes and the homogeneous Monge-Amp\`ere equation, \emph{J. Diff.Geom.} {\bf 34} (1991), 561-570.

\bibitem{gs2} V.W. Guillemin and M. Stenzel, Grauert tubes and the homogeneous Monge-Amp\`ere equation II, \emph{J. Diff. Geom.} {\bf 35} (1992), 627-641.

\bibitem{ls} L. Lempert and R. Sz\H{o}ke, Global solutions of the homogeneous Monge-Amp\`ere equation and complex structures on the tangent bundle of Riemannian manifolds, \emph{Math. Ann.} {\bf 290} (1991), 689-712.

\bibitem{o} B. O'Neill, Semi-Riemannian Geometry With Applications to Relativity. Pure and Applied Mathematics, 103. Academic Press, Inc. New York, 1983.

\bibitem{s} R. Sz\H{o}ke, Complex structures on tangent bundles of Riemannian manifolds, \emph{Math. Ann.} {\bf 291} (1991), 409-428.

\bibitem{s3} R. Sz\H{o}ke, Automorphisms of certain Stein manifolds, \emph{Math. Z.} {\bf 219} (1995), 357-385.

\bibitem{s2} R. Sz\H{o}ke, Adapted complex structures and Riemannian homogeneous spaces, \emph{Ann. Polon. Math.} {\bf 70} (1998), 215-220.


\end{thebibliography}
\end{document}